\documentclass[11pt]{amsart}

\usepackage{hyperref}

\usepackage{graphicx, enumerate, url}
\usepackage{amssymb,mathtools}
\usepackage{algorithm}
\usepackage{algpseudocode}
\usepackage{color}
\usepackage{tikz}
\usetikzlibrary{3d,calc}

\numberwithin{equation}{section}
\numberwithin{algorithm}{section}

\theoremstyle{plain}
\newtheorem{theorem}{Theorem}[section]

\newtheorem{lemma}[theorem]{Lemma}
\newtheorem{corollary}[theorem]{Corollary}

\theoremstyle{definition}

\newtheorem{problem}[theorem]{Problem}

\theoremstyle{remark}

\DeclareMathOperator{\V}{V}

\DeclareMathOperator{\argmax}{arg\,max}
\DeclareMathOperator{\argmin}{arg\,min}

\newcommand{\N}{\mathbb{N}}

\newcommand{\R}{\mathbb{R}}
\newcommand{\Z}{\mathbb{Z}}
\newcommand{\s}{\mathbf{s}}

\newcommand{\x}{\mathbf{x}}

\newcommand{\y}{\mathbf{y}}
\newcommand{\z}{\mathbf{z}}
\newcommand{\U}{\mathbf{U}}
\renewcommand{\V}{\mathbf{V}}
\newcommand{\W}{\mathbf{W}}
\newcommand{\cA}{\mathcal{A}}
\newcommand{\cB}{\mathcal{B}}
\newcommand{\cC}{\mathcal{C}}

\newcommand{\ba}{\mathbf{a}}
\newcommand{\bb}{\mathbf{b}}
\newcommand{\bc}{\mathbf{c}}

\newcommand{\bp}{\mathbf{p}}
\newcommand{\br}{\mathbf{r}}
\newcommand{\bu}{\mathbf{u}}
\newcommand{\bv}{\mathbf{v}}
\newcommand{\bw}{\mathbf{w}}
\newcommand{\bx}{\mathbf{x}}
\newcommand{\by}{\mathbf{y}}

\newcommand{\rB}{\mathrm{B}}
\newcommand{\rC}{\mathrm{C}}

\newcommand{\rL}{\mathrm{L}}
\newcommand{\rS}{\mathrm{S}}
\newcommand{\0}{\mathbf{0}}
\newcommand{\1}{\mathbf{1}}
\newcommand{\trans}{^\top}

\newcommand{\et}{\mbox{\boldmath{$\eta$}}}

\begin{document}
\title{Partial minimization of strict convex functions and tensor scaling}
\author[Shmuel Friedland]{Shmuel~Friedland}
\address{Department of Mathematics and Computer Science, University of Illinois at Chicago, Chicago, Illinois, 60607-7045, USA }
\email{friedlan@uic.edu}
\date{June 4, 2019}
\begin{abstract}  Assume that $f\in\rC^2(\R^n)$ is a strict convex function with a unique minimum.  We divide the vector of $n$ variables to $d\ge 2$ groups of vector  subvariables.  We assume that we can find the partial minimum of $f$  with respect to each vector subvariable while other variables are fixed.  We then describe an algorithm that partially minimizes each time on a specifically chosen vector subvariable.  This algorithm converges geometrically to the unique minimum.  The rate of convergence depends on the uniform bounds on the eigenvalues of the Hessian of $f$ in the compact sublevel set $f(\x)\le f(\x_0)$, where $\x_0$ is the starting point of the algorithm.  In the case where $f(\x)=\x^\top A\x+\bb^\top \x$ and $d=n$ our method can be considered as a generalization of the classical conjugate gradient method.
 The main result of this paper is the observation that the celebrated Sinkhorn diagonal scaling algorithm for matrices, and the corresponding  diagonal scaling of tensors, can be viewed as partial minimization of certain logconvex functions.
\end{abstract}
\maketitle


\maketitle

\section{Introduction} \label{sec:intro} 
Let $f\in\rC^2(\R^n)$ is a strict convex function, that is, the Hessian $H(f)(\x)$ is positive definite for each $\x\in\R^n$.  We assume that $f$ has a minimum at $\x^\star\in\R^n$, which is necessary unique.  It is well known that a necessary and sufficient condition for the existence of $\x^\star$ is:
\begin{eqnarray}\label{maxinfcond}
\lim_{\|\x\|\to \infty} f(\x)=\infty.
\end{eqnarray}
See Lemma \ref{existminlem}.  We now recall the notion of partial minimization of $f$.
For $m\in\N$ denote $[m]=\{1,\ldots,m\}\subset \N$.
Divide the vector $\x=(x_1,\ldots,x_n)^\top$ to $d\ge 2$ groups: $\x^\top=(\x_1^\top,\ldots,\x_d^\top)$, where  $\x_i\in \R^{m_i}$ for $i\in[d]$ and $\sum_{i=1}^p m_i=n$.  (Thus $d\in[n]\setminus\{1\}$.)  View $\x$ as $(\x^j,\x_j)$ where $\x^j\in\R^{n-m_j}$ is obtained from $\x$ by deleting the vector coordinate $\x_j$.  Denote by $\nabla_j f(\x)\in \R^{m_j}$ the vector of derivatives of $f(\x)$ with respect to the coordinates in $\x_j$.
Minimize $f(\x)$ with respect to the variable $\x_j$ while keeping all other variable fixed:
\begin{eqnarray}\label{partmin}
\min\{f(\x), \x_j\in\R^{m_j}, \x=(\x^j,\x_j)\}=f(\x^j, \x_j(\x^j)).
\end{eqnarray}
Our main assumption is that we can find $\x_j(\x^j)$ either precisely, or with a prescribed accuracy.  This assumption holds if $f$ is a polynomial of degree $2$ $f(\x)=\x^\top A\x+\bb^\top\x+c$, where $A$ is a symmetric positive definite matrix and $d=n$.  This is the classical case of the conjugate gradient \cite{HS52}.
The main point of this paper is to show that this assumption holds if we consider the classical scaling algorithm of Sinkhorn \cite{Sin64}, or more general tensor scaling problem \cite{Bap82, Rag84, FL89, Fri11}.  Matrix scaling problems arise in several areas of applied and pure mathematics.  There are many available algorithms to achieve the scaling.  See \cite{ALOW17} for a historical survey and for new suggested algorithms.  The main purpose of this paper to show that matrix and tensor scaling could be efficiently implemented  using our simple algorithm which ensures geometric convergence.  While for matrices our algorithm reduces to alternating scaling, for tensors the algorithm chooses the order of scaling.

We now state briefly our algorithm: 

 \noindent
 $\quad\quad$ $\;\;$\textbf{Algorithm} 
\begin{itemize}
\item[] Choose $\x_0\in\R^n$. 
\item[] for $k:=0,1,2, \ldots$
\item[] $\quad$$j\in\argmax\{\|\nabla_l f(\x_k)\|, l\in[d]\}$
\item[] $\quad$$\x_{k+1}=(\x_k^j, \x_j(\x_k^j))$
\item[]end
\end{itemize} 
We show that this algorithm converges geometrically to $\x^\star$ with at least a factor $(1-\frac{\alpha}{\sqrt{d-1}\beta})$, where $\alpha$ and $\beta$ are the minimum and the maximum of the lowest and highest eigenvalues of $H(f)$ respectively in the compact convex sublevel region $\{\x, f(\x)\le f(\x_0)\}$.  

Note that if $d=2$, i.e., $\x=(\x_1,\x_2)$, then after one iteration the above minimization algorithm is an alternating minimization, as in the Sinkhorn algorithm.  Instead of using the standard  coordinates $\x=(x_1,\ldots,x_n)^\top$ we can use the coordinates $\hat\x=P\x$, where the $n$ rows  of $P$: $\bp_1^\top, \ldots,\bp_n^\top$ are linearly independent.
In the conjugate gradient algorithm we need to choose the vectors $\bp_1,\ldots,\bp_n$ to be orthogonal with respect to $A$: $\bp_i^\top A\bp_j=0$ for $i\ne j$ \cite{HS52}.

We now explain briefly why Sinkhorn scaling algorithm for matrices can be stated as a partial minimization of strict convex function.  For simplicity of exposition ourselves mainly to positive rectangular matrices $B=[b_{i,j}]\in\R^{l\times m}$.  For $\bu=(u_1,\ldots,u_l)^\top \in\R^l$ we denote by  $D(\bu)\in\R^{l\times l}$ the diagonal matrix with the diagonal entries $e^{u_1},\ldots, e^{u_l}$. Let $\1_n=(1,\ldots,1)\trans \in\R^n$ and assume  $\br=(r_1,\ldots,r_l)^\top, \bc=(c_1,\ldots,c_m)^\top$ are given positive vectors satisfying $\1_l^\top \br=\1_m^\top\bc$.  The scaling problem is finding $\bu,\bv$ such that the matrix $D(\bu)B D(\bv)$ has rows and column sums $\br$ and $\bc$ respectively:
\begin{eqnarray}\label{matresprb}
D(\bu) B D(\bv)\1_m=\br,\quad \1_l^\top D(\bu) B D(\bv)=\bc^\top, 
\end{eqnarray}
for some $\bu\in\R^l,\bv\in\R^m$.
Clearly, this problem is equivalent to the scaling problem when we replace $\br, \bc$ with $b\br,b\bc$ for some positive $b>0$.  For a given nonzero vector $\bw\in \R^n$ denote by $\rL(\bw)=\{\x\in\R^n, \bw^\top \x=0\}$.  The the dimension of $\rL(\bw)$ is $n-1$ and we identify $\rL(\bw)$ with $\R^{n-1}$.  Let 
\begin{eqnarray*}
f(\bu, \bv)=\sum_{i=j=1}^{l,m} b_{i,j}e^{u_i+v_j}.
\end{eqnarray*}
Clearly, $f(\x), \x=(\bu,\bv)$ is a convex function on $\R^{l+m}$.  We consider the restriction of $f$ to $L(\br)\times L(\bc)$. Since $B>0$ it follows that $f(\x)$ is strictly convex on $L(\br)\times L(\bc)$ and the condition \eqref{maxinfcond} holds, see Section \ref{sec:tenscalgo}.  Let $\x^\star=(\bu^\star, \bv^\star)\in \rL(\br)\times \rL(\bc)$ be the minimum point of $f| \rL(\br)\times \rL(\bc)$.  Use Lagrange multipliers to deduce that $D(\bu^\star)B D(\bv^\star)$ has row and column sums $b\br, b\bc$ for some $b>0$.
Fix $\bv\in \rL(\bc)$ and find partial minimum of $\min\{f(\bu,\bv), \bu\in \rL(\br)\}$.  Use Lagrange multipliers to deduce that this minimum is achieved at unique $\bu(\bv)$ such that the row sums of $D(\bu(\bv)) B D(\bv)$ are of the form $b\br$.
We now give a simple formula for $\bv$.  Observe first that the equality $D(\bu)BD(\bv)\1_m=\br$ is uniquely solvable by
$\tilde u_i=\log r_i - \log (BD(\bv)\1_m)_i$ for $i\in[l]$.  Let $\tilde \bu(\bv)=(\tilde u_1,\ldots,\tilde u_l)^\top$.   Note that $\tilde\bu(\bv)$ is the scaling part of Sinkhorn algorithm.
Then $\bu(\bv)=\tilde \bu(\bv)-a\1_l$, where $a=\br^\top\tilde\bu(\bv)/(\br^\top \1_l)$.
Similarly, for a fixed $\bu\in\rL(\br)$ the minimum of $f(\bu,\bv)$ for $\bv\in\rL(\bc)$ is achieved for unique $\bv(\bu)$ which can be obtained as follows. First by use Sinkhorn scaling to $D(\bu)BD(\tilde\bv(\bu))$ to have the column sum $\bc$.  Second let $\bv(\bu)=\tilde\bv(\bu) -(\bc^\top\tilde\bv(\bu)/\bc^\top\1_m)\1_m$.
Since $d=2$ the partial minimization algorithm is completely equivalent to Sinkhorn minimization algorithm.  The geometric rate of convergence depends on the estimates of the eigenvalues of Hessian on the sublevel set $f(\x)\le f(\x_0)$ in $\rL(\br)\times \rL(\bc)$.  

In the case where $B$ has some zero entires then the scaling problem is solvable if and only if there exist a nonnegative matrix $C=[c_{i,j}]\in\R^{l\times m}$ with the same $0$ pattern as $B$, ($b_{i,j}=0\iff c_{i,j}=0$), and with the row and column sums $\br,\bc$ \cite{Men68}.
The existence of such $C$ is a linear programming problem that can be solved in polynomial time \cite{Kar,Kha,Fri11}.  If $B$ can be scaled, it is possible to convert the scaling problem to partial minimization of $f(\x)$ on a corresponding subspace of $L\subset \rL(\br)\times \rL(\bc)$.

We now summarize the contents of the paper. In Section \ref{sec:convalg} we show that our algorithm converges geometrically to $\x^\star$: the unique minimum point of $f$.  Denote by $V(t)=\{\x\in\R^n, f(\x)\le t\}$
the compact convex sublevel set corresponding to $t\ge t^\star=f(\x^\star)$.
Let $0<\alpha(t)\le \beta(t)$ be the minimum and the maximum of the smallest and the biggest eigenvalues of the Hessian $H(f)$ in $V(t)$.  Let $\kappa(t)=\frac{\beta(t)}{\alpha(t)}$. Set $t_k=f(\x_k)$, where $\x_k$ are given by our algorithm.  Then $t_k$ is a strictly decreasing sequence which converges to $t^\star$, unless the algorithm reaches $\x^\star$ in a finite number of steps .
Theorem \ref{convalgo}   shows that the rate of convergence of $\x_k$ to $\x^\star$ and $t_k$ to $t^\star$ is at least of order $(1-\frac{1}{(d-1)\kappa(t_0)})^{k-1}$.  More precisely, 
\begin{eqnarray*}
&&|t_k-t^\star|\le  \frac{\|\nabla f(\x_1)\|^2}{\alpha^2(t_1)}\prod_{q=1}^{k-1}(1-\frac{1}{(d-1)\kappa(t_q)}), \\
&&\|\x_k-\x^\star\|^2\le \frac{\|\nabla f(\x_1)\|^2}{\alpha(t_1)\alpha(t_k)}\prod_{q=1}^{k-1}(1-\frac{1}{(d-1)\kappa(t_q)}).
\end{eqnarray*}

In Section \ref{sec:tenres} we recall our results on  tensor scaling \cite{Fri11}.  Assume that $\cB=[b_{i_1,\ldots,i_d}]\in \R^{m_1}\times \ldots\times \R^{m_d}$ is a given nonnegative $d$-mode tensor.  Let $\bx=(\x_1,\ldots,\x_d)$, where $\x_j=(x_{j,1},\ldots,x_{d,m_j})^\top\in\R^{m_j}$. A scaling of $\cB$ is the tensor $\cB(\x)=[e^{x_{1,i_1}+\ldots+x_{d,i_d}}b_{i_1,\ldots,i_d}]$.  Let $\s_j$ be positive probability vectors in $\R^{m_j}$ for $j\in[d]$.  Then the scaling problem is to find $\x$ such that the $j$-th slice sum of $\cB(\x)$, obtained by summing on the indices $i_1,\ldots,i_{j-1},i_{j+1},\ldots, i_d$, is $\s_j$ for each $j\in[d]$.  If $\cB$ is positive then
such scaling exists.  If $\cB$ has zero entries then such scaling exists if and only if there exists a nonnegative tensor $\cC$ with the same $0$ pattern as $\cB$ and with the sum slices $\s_1,\ldots,\s_d$ \cite{BR89,FL89, Fri11}.  We show that if scaling of $\cB$ exists then it can be achieved by finding the minimum of the strict convex function $f$ on a subspace $\rL\subset \rL(\s_1)\times\ldots\times \rL(\s_d)$.  

In Section \ref{sec:tenscalgo} we discuss the application of our algorithm to tensor scaling.  In the case where $\cB$ positive, or more general, where the strict convex function
$f$ is defined on the whole  $\rL(\s_1)\times\ldots\times \rL(\s_d)$, our algorithm applies straightforward.  For matrices, $d=2$ it is exactly the Sinkhorn scaling algorithm, which was explained above.  In the case of tensors, $d\ge 3$, the algorithm chooses each time the scaling slice.  In the case where $f$ is strictly convex on a subspace $\rL\subset \rL(\s_1)\times\ldots\times \rL(\s_d)$, we describe a simple modification of our algorithm and justify its geometric  convergence.

In Section \ref{sec:sbp} we show that our algorithm applies also to a generalized discrete Schr\"odinger's bridge problem.  (The discrete Schr\"odinger's bridge problem is a scaling of a given column stochastic matrix to another column stochastic matrix  $B$ so that $B\ba=\bb$, where $\ba,\bb$ are two given positive probabiitiy vectors \cite{GP15, Fri17}.)

\section{The convergence of the algorithm}\label{sec:convalg}
\begin{lemma}\label{existminlem}
Let $f\in\rC^2(\R^n)$ be strictly convex.  Then the following conditions are equivalent:
\begin{enumerate}
\item The function $f$ has a unique minimum $\x^\star\in\R^n$.
\item The condition  \eqref{maxinfcond} holds.
\end{enumerate}
\end{lemma}
\begin{proof}  (1)$\Rightarrow$(2).  Let $\rS^{n-1}$ be the $n-1$ dimensional sphere $\|\y-\x^\star\|=1$.    
Fix $\y\in \rS^{n-1}$.  Consider the strict convex function in one variable: $g_{\y}(t)=f(\x^\star+t(\y-\x^\star))$.  Then $g_{\y}'(0)=0$ and $g_{\y}'(1)=\nabla f(\y)^\top (\y-\x^\star)>0$.  Let $\nu=\min\{g_{\y}'(1), \y\in\rS^{n-1}\}$.  Clearly, $\nu>0$.  As $g_{\y}'(t)$ increases for $t>0$ it follows that $g_{\y}'(t)\ge g'_{\y}(1)$ for $t\ge 1$.  In particular, 
\begin{eqnarray*}
g_{\y}(t)\ge g_{\y}(1)+g_{\x}'(1)(t-1)\ge f(\x^\star) +\nu(t-1) \textrm{ for } t\ge 1
\end{eqnarray*}
Hence $f(\x)\ge f(\x^\star) + \nu(\|\x\|-1)$ if $\|\x-\x^\star\|\ge 1$.  This inequality yields \eqref{maxinfcond}.

\noindent (2)$\Rightarrow$(1)  Fix $\x_0\in\R^n$.  Then there exists $r>0$ such that $\min\{f(\x), \|\x-\x_0\|=r\}>f(\x_0)$.  Let $\min\{f(\x), \|\x-\x_0\|\le r\}=f(\x^\star)$.  Clearly, $\|\x^\star -\x_0\|<r$.  Therefore $\nabla f(\x^\star)=\0$.  As $f(\x)$ is convex we deduce that $f(\x)\ge f(\x_0)$ for each $\x\in\R^n$.  As $f(\x)$ is strictly convex $\x^\star$ is the unique point of minimum of $f$.
\end{proof}
Note that the function $f(x)=e^x, x\in\R$ is strictly convex on $\R$ but $f(x)$ does not have a minimum on $\R$.

In what follows we assume that $f\in\rC^2(\R^n)$ is strictly convex and $\x^\star$ is the unique minimum point of $f$.  Then for each $\x\in\R^n\setminus{\x^\star}$ the  sublevel set
\begin{eqnarray*}
V(t)=\{\y\in\R^n, f(\y)\le t\}, \quad t=f(\x)
\end{eqnarray*}
is a compact strictly convex set, with a $\rC^2$ boundary $\partial V(t)$, with an interior containing $\x^\star$.  Let $t^\star=f(\x^\star)$.  Then $V(t^\star)=\{\x^\star\}$.
Thus $\R^n\setminus \{\x^\star\}$ is parametrized by $\partial V(t), t>t^\star$.

Fix $t_0=f(\x_0)>t^\star$.  Then $f$ is uniformly strictly convex in $V(t_0)$: The eigenvalues of $H(f)(\x), \x\in V(t_0)$ are in a fixed interval $[\alpha(t_0),\beta(t_0)]$ for some $0<\alpha(t_0) \le \beta(t_0)$. 
Thus for each $\x,\y\in V(t_0)$ we have the inequalities:
\begin{eqnarray}\label{strconvond1}
f(\x)+\nabla f(\x)^\top(\y-\x) +\frac{\alpha(t_0)}{2} \|\y-\x\|^2\le f(\y)\le\\ f(\x)+\nabla f(\x)^\top(\y-\x) +\frac{\beta(t_0)}{2} \|\y-\x\|^2.
\label{strconvond2}
\end{eqnarray}
In particular, for $ \x\in V(t_0)$ we have
\begin{eqnarray}\label{lowupbndf}
f(\x^\star)+ \frac{\alpha(t_0)}{2} \|\x-\x^\star\|^2\le f(\x)\le f(\x^\star) +\frac{\beta(t_0)}{2} \|\x-\x^\star\|^2.
\end{eqnarray}
Denote by $\rB(\x, R^2)$ the closed ball $\{\by\in\R^n,\|\bx-\by\|^2\le R^2\}$.  
Let $\kappa(t_0)=\frac{\beta(t_0)}{\alpha(t_0)}$ and define
\begin{eqnarray}\label{defxplusa}
\x^+=\x-\frac{1}{\beta(t_0)}\nabla f(\x), \quad \x^{++}=\x-\frac{1}{\alpha(t_0)}\nabla f(\x).
\end{eqnarray}
In what follows we need the following lemma:
\begin{lemma}\label{xflem}  Assume that $\x\in V(t_0)$.  Let 
\begin{eqnarray}\label{xflem1}
\x^{a}=\x-\frac{2}{\beta(t_0)}\nabla f(\x).
\end{eqnarray}
Then 
\begin{enumerate}
\item 
$f(\x^a)\le f(\x)$.  
\item  $[\x,\x^a]\subset V(t_0)$.
\item  $f(\x)-f(\x^\star)\ge f(\x)-f(\x^+)\ge \frac{\|\nabla f(\x)\|^2}{2\beta(t_0)}$.
\end{enumerate}
\end{lemma}
\begin{proof}  (1) If $\nabla f(\x)=\0$, i.e., $\x=\x^\star$ the (1) trivially holds. Suppose that  $\nabla f(\x)\ne\0$ and assume to the contrary that $f(\x^a)>f(\x)$.  Let $h(t)=f(\x-t\nabla f(\x))$.  Then $h'(0)=-\|\nabla f(\x)\|^2$.  Recall that $h(t)$ is a strict convex function.  Hence there exists $t_1\in (0, \frac{2}{\beta(t_0)})$ such that $h'(t_1)=0$ and $h'(t)>0$ for $t>t_1$.  Thus there exists $t_2\in (t_1, \frac{2}{\beta(t_0)})$ such that $f(\y)=f(\x)$ for $\y=\x-t_2\nabla f(\x))$.  Note that $\y\in V(t_0)$.  This contradicts the inequality \eqref{strconvond2}.

\noindent (2)  As $f(\x^a)\le f(\x)\le t_0$ the convexity of $f$ yields that the interval $[\x,\x^a]$ is in $V(t_0)$.

\noindent(3) Clearly $\x^+=\frac{1}{2}(\x +\x^a)\in [\x,\x^a]$.  Hence 
\begin{eqnarray}  \label{upbetineq}&& f(\x^+)\le \\
&&f(\x)+\nabla f(\x)^\top (\x^+-\x)+\frac{\beta(t_0)}{2}\|\x^+-\x\|^2=f(\x)-\frac{\|\nabla f(\x)\|^2}{2\beta(t_0)}.\notag
\end{eqnarray} 
Therefore (3) holds.
\end{proof}
We now bring the following simple lemma which is basically in \cite{BLS15}:
\begin{lemma}\label{lem1}  Assume that $f\in\rC^2(\R^n)$ is strictly convex and $\x^\star$ is the unique minimum point of $f$.  Fix $\x\in V(t_0)$ and assume that $\x^\star\in \rB(\x, R_0^2)$.   Then we can choose $R_0=R(\x)$ and the following conditions hold:
\begin{eqnarray}\label{eq1}
 R(\x)^2= \frac{2}{\alpha(t_0)}(f(\x)-f(\x^\star))\le\frac{\|\nabla f(\x)\|^2}{\alpha^2(t_0)},\\
\label{eq2}
\x^\star\in \rB(\x^{++}, \frac{\|\nabla f(\x)\|^2}{\alpha^2(t_0)}-\frac{2}{\alpha(t_0)}(f(\x)-f(\x^\star))\subseteq\\
\notag
 \rB(\x^{++}, \frac{\|\nabla f(\x)\|^2}{\alpha^2(t_0)}(1-\frac{1}{\kappa(t_0)})-\frac{2}{\alpha(t_0)}(f(\x^+)-f(\x^\star)),\\
 \label{eq3}
 \frac{\|\nabla f(\x)\|}{\beta(t_0)}\le \|\x-\x^\star\|.
\end{eqnarray}
\end{lemma}
\begin{proof}  As $\x\in V(t_0)$ the left hand side of \eqref{lowupbndf} yields that $\x^\star\in\rB(\x,R(\x)^2)$, where $\R(\x)^2$ is given by \eqref{eq1}.  Clearly
\begin{eqnarray*}
\|\x^\star -\x^{++}\|^2=\|(\x^\star -\x+\frac{1}{\alpha(t_0)}\nabla f(\x)\|^2=\\
\|(\x^\star -\x\|^2 +
\frac{2}{\alpha(t_0)}\nabla f(\x)^\top (\x^\star-\x) +\frac{\|\nabla f(\x)\|^2}{\alpha^2(t_0)}.
\end{eqnarray*}
As $\x^\star,\x\in V(t_0)$ \eqref{strconvond1} yields:
\begin{eqnarray}\label{convcondxs}
f(\x^\star)\ge f(\x)+\nabla f(\x)\trans(\x^\star-\x)+\frac{\alpha(t_0)}{2}\|\x^\star - \x\|^2.
\end{eqnarray}
Thus
\begin{eqnarray*}
\|\x^\star-\x^{++}\|^2\le 
\frac{\|\nabla f(\x)\|^2}{\alpha^2(t_0)} -\frac{2}{\alpha(t_0)}(f(\x)-f(\x^\star)).
\end{eqnarray*}
This proves the first part of  \eqref{eq2}.  Hence the inequality in \eqref{eq1} holds.
Use part (3) of Lemma \ref{xflem}
to replace $f(\x)$ in the first part of \eqref{eq2} by a smaller quantity $f(\x^+)+\frac{\|\nabla f(\x)\|^2}{2\beta(t_0)}$ to obtain the second part of \eqref{eq2}.  

Combine \eqref{upbetineq} with \eqref{lowupbndf} to deduce 
\begin{eqnarray*}
\frac{\|\nabla f(\x)\|^2}{2\beta(t_0)}\le f(\x)-f(\x^+)\le f(\x)-f(\x^\star)\le \frac{\beta(t_0)}{2}\|\x-\x^\star\|^2. 
\end{eqnarray*}
This show the inequality \eqref{eq3}.  
\end{proof}
We now show that in our algorithm the sequences $\x_k,f(\x_k), k\in\N$ converge geometrically to $\x^\star,f(\x^\star)$ respectively:
\begin{theorem}\label{convalgo}  
Assume that $f\in\rC^{2}(\R^n)$ is a strict convex function which has a unique minimum point $\x^\star$.  Let $\x_0\in\R^n$ and  $\x_k,k\in\N$ be given by our algorithm.   Set $t_k=f(\x_k)$ for $k\in\Z_+$.  Assume that the eigenvalues of $H(f)(\x), \x\in V(t_k)$ are in the minimal interval $[\alpha(t_k),\beta(t_k)]$, where $0<\alpha(t_k) \le \beta(t_k)$. Denote $\kappa(t_k)=\frac{\beta(t_k)}{\alpha(t_k)}$.  
\begin{enumerate}
\item
If $\x_{k-1}\ne \x^\star$ for some $k\in\N$ then $t_{k-1}>t_k$.
\item The sequences $\{t_k\},\{\beta(t_k)\},\{-\alpha(t_k)\},\{\kappa(t_k)\}, k\in\Z_+$ are nonincreasing sequences which converge to $t^\star,\beta(t^\star),-\alpha(t^\star),\kappa(t^\star)$ respectively.
\item
For each $k\in\N$ the following inequalities hold:
\begin{eqnarray}\notag
&& f(\x_k)-f(\x^\star)\le\\ \label{convalgo1}
 && (f(\x_0)-f(\x^\star))(1-\frac{1}{d\kappa(t_0)})\prod_{i=1}^{k-1}(1-\frac{1}{(d-1)\kappa(t_i)})\le \\  \label{convalgo3}
 && \frac{\|\nabla f(\x_0)\|^2}{2\alpha(t_0)}(1-\frac{1}{d\kappa(t_0)})\prod_{i=1}^{k-1}(1-\frac{1}{(d-1)\kappa(t_i)})\le \\ \notag
&&\frac{\|\nabla f(\x_0)\|^2}{2\alpha(t_0)}(1-\frac{1}{d\kappa(t_0)})(1-\frac{1}{(d-1)\kappa(t_0)})^{k-1},\\
\label{convalgo2}
&&\|\x_k-\x^\star\|^2\le \frac{\|\nabla f(\x_0)\|^2}{\alpha(t_k)\alpha(t_0)}(1-\frac{1}{d\kappa(t_0)})\prod_{i=1}^{k-1}(1-\frac{1}{(d-1)\kappa(t_i)}).
\end{eqnarray}
\end{enumerate}
\end{theorem}
\begin{proof}  Note 
\begin{eqnarray*}
\|\nabla f(\x)\|^2=\sum_{l=1}^d \|\nabla_l f(\x)\|^2\Rightarrow
\max\{\|\nabla_l f(\x)\|, l\in[d]\}\ge \frac{\|\nabla f(\x)\|}{\sqrt{d}}.
\end{eqnarray*}
(1) Clearly if $\x_{k-1}=\x^\star$ then $\x_p=\x^\star$ for $p\ge k$.
Assume that $\x_{k-1}\ne \x^\star$.  Then $\|\nabla f(\x_{k-1})\|>0$.  
Let $j_{k-1}\in\argmax \{\|\nabla_l f(\x_{k-1})\|, l\in[d]\}$.  Then $\|\nabla_{j_{k-1}} f(\x_{k-1})\|>0$.  Hence $t_{k-1}>t_k$.

\noindent
(2) As $\{t_k\},k\in\Z_+$ is a nonincreasing sequence we deduce that $V(t_k)\subseteq V(t_{k-1})$ for $k\in\N$.  Hence the sequence $\{\alpha(t_k)\},k\in\N$ is a nonincreasing, and the sequences $\{\beta(t_k)\},k\in\N$ and $\{\kappa(t_k)\},k\in\N$
are nondecreasing.  The equality $\lim_{k\to\infty} t_k=t^\star$ follows from \eqref{convalgo1}.   The inequality \eqref{convalgo2} yields
\begin{eqnarray*}
\lim_{k\to\infty} \x_k=\x^\star,\lim_{k\to\infty}\alpha(t_k)=\alpha(t^\star),\lim_{k\to\infty}\beta(t_k)=\beta(t^\star),\lim_{k\to\infty}\kappa(t_k)=\kappa(t^\star).
\end{eqnarray*}

\noindent
(3) First we show the inequality \eqref{convalgo1} for $k=1$.  Assume that $j_0\in\argmax\{ \|\nabla_l f(\x_0)\|, l\in[d]\}$.  Hence $\|\nabla _{j_0} f(\x_0)\|\ge \frac{\|\nabla f(\x_0)\|}{\sqrt{d}}$.  Let $g(\x_{j_0})=f(\x_0^{j_0},\x_{j_0})$, where $\x_0=(\x_0^{j_0},\x_{j_{0},0})$.  Thus $g$ is a strictly convex function, whose Hessian is a submatrix of the Hessian of $f$.
Hence the eigenvalues of the Hessian of $g$ are also in the interval $[\alpha(t_0),\beta(t_0)]$.  Recall that $\argmin g=\x_{j_0}^\star=\x_j(\x_0^{j_0})$.  Then $\x_1=(\x_0^{j_0},\x_{j_0}^\star)$.  We now estimate from below $g(\x_{j_0,0})-g(\x_{j_0}^\star)$. The lower bound (3) of Lemma  \eqref{xflem} yields:
\begin{eqnarray*}
f(\x_0)-f(\x_1)=g(\x_{j_0,0})-g(\x_{j_0}^\star)\ge \\ \frac{\|\nabla g(\x_{j_0,0}\|^2}{2\beta(t_0)}= \frac{\|\nabla f_{j_0}(\x_0)\|^2}{2\beta(t_0)}\ge \frac{\|\nabla f(\x_0)\|^2}{2\beta (t_0)d}.
\end{eqnarray*}
The inequality \eqref{eq1} yields $f(\x_0)-f(\x^\star)\le \frac{\|\nabla f(x_0)\|^2}{2\alpha(t_0)}$. Assuming that $f(\x_0) >f(\x^\star)$ we obtain
\begin{eqnarray*}
&&\frac{f(\x_1)-f(\x^\star)}{f(\x_0)-f(\x^\star)}=1 - \frac{f(\x_0)-f(\x_1)}{f(\x_0)-f(\x^\star)}\le\\ 
&&1 -  \big(\frac{ \|\nabla f(\x_0)\|^2}{2\beta(t_0) d}\big)/\big(\frac{\|\nabla f(\x_0)\|^2} {2\alpha(t_0)}\big)=1-\frac{1}{d\kappa(t_0)}.
\end{eqnarray*}
This proves the first inequality in \eqref{convalgo1} for $k=1$.  

Assume now that $k=2$.  The definition of $\x_1$ yields that $\nabla_{j_0} f(\x_1)=0$.
Hence $\max\{\|\nabla_l f(\x_1)\|, l\in[d]\}\ge \frac{\|\nabla f(\x_1)\|}{\sqrt{d-1}}$.
Use the same arguments as above to show that $f(\x_2)-f(\x^\star)\le (f(\x_1)-f(\x^\star))(1-\frac{1}{(d-1)\kappa(t_1)})$. Hence \eqref{convalgo1} holds for $k=2$.  Similarly, the inequality  \eqref{convalgo1} holds for each $k\ge 2$.

Use the inequality \eqref{eq1} to deduce the inequality in \eqref{convalgo3}.
As $\kappa(t_k)\le \kappa(t_0)$ for each $k\in\N$ we deduce the inequality below \eqref{convalgo3}.
According to Lemma \ref{lem1} $\x^\star\in \rB(\x_k,R^2(\x_k))$.   Use \eqref{convalgo3} to deduce \eqref{convalgo2}.
\end{proof}

Observe that  our algorithm is an alternating algorithm for $d=2$ after the first step.
 \section{The tensor scaling problem}\label{sec:tenres}
 In this section we first recall briefly the results in \cite{Fri11} that we need.
 For positive integers $d,m_1,\ldots,m_d$
 denote by $\R^{m_1\times\ldots\times m_d}$ the linear space $d$-mode tensors
 $\cA=[a_{i_1,i_2,\ldots,i_{d}}], i_j\in [m_j], j\in[d]$.  Note that a $1$-mode tensor is a vector, and a $2$-mode tensor is a matrix.
 Assume that $d\ge 2$.  For a fixed $i_k\in[m_k]$ the $(d-1)$-mode tensor $[a_{i_1,\ldots,i_d}], i_j\in[m_j,j\in [d]\backslash
 \{k\}$ is called the $(k,i_k)$ \emph{slice} of $\cA$.  
 For $d=2$ the $(1,i)$ slice and the $(2,j)$ slice are the $i-th$ row and the $j-th$ column of a given matrix.  In the rest of the paper we assume:
 \begin{eqnarray}\label{admjas}
 d\ge 2, \quad m_j\ge 2 \textrm{ for } j\in[d].
 \end{eqnarray}
 Let
 \begin{equation}\label{kikslicesum}
 s_{k,i_k}:=\sum_{i_j\in[m_j],j\in [d]\backslash\{k\}} a_{i_1,\ldots,i_d}, \; i_k\in[m_k],k\in[d]
 \end{equation}
 be the $(k,i_k)$-slice sum.  Denote
 \begin{equation}\label{kikslsumvec}
 \s_k:=(s_{k,1}, \ldots,s_{k,m_k})\trans, \quad k\in[d]
 \end{equation}
 the $k$-slice sum.  Note that $k$-slice sums satisfy the compatibility conditions
 \begin{equation}\label{compslccon}
 \sum_{i_1=1}^{m_1} s_{1,i_1}=\ldots=\sum_{i_d=1}^{m_d} s_{d,i_d}.
 \end{equation}

 Two $d$-mode tensors $\cA=[a_{i_1,i_2,\ldots,i_{d}}],\cB=[b_{i_1,i_2,\ldots,i_{d}}]\in
 \R^{m_1\times\ldots\times m_d}$
 are called \emph{positive diagonally} equivalent if there exist $\x_k=(x_{k,1},\ldots,
 x_{k,m_k})\trans\in\R^{m_k}, k\in[d]$
 such that
 $a_{i_1,\ldots,i_d}=e^{x_{1,i_1}+\ldots+x_{d,i_d}}b_{i_1,\ldots,i_d}$ for all $i_j\in[m_j]$
 and $j\in[d]$.
 Denote by  $\R_+^{m_1\times\ldots\times m_d}$ the cone of nonnegative,(entrywise),  $d$-mode tensors.

 We assume that $\cB=[b_{i_1,i_2,\ldots,i_{d}}]\in
 \R_+^{m_1\times \ldots\times m_d}$ is a given nonnegative tensor with no zero slice $(k,i_k)$.
 Let $\s_k\in \R_+^{m_k}, k\in[d]$ are given $k$ positive vectors satisfying the conditions
 (\ref{compslccon}).
 Denote by $\R_+^{m_1\times \ldots\times m_d}(\cB, \s_1,\ldots,\s_d)$ the set of all nonnegative
 $\cA=[a_{i_1,i_2,\ldots,i_{d}}]\in
 \R_+^{m_1\times \ldots m_d}$ having the same zero pattern as $\cB$, i.e. $a_{i_1,\ldots,i_d}=0
 \iff b_{i_1,\ldots,i_d}=0$
 for all indices $i_1,\ldots,i_d$, and satisfying the condition (\ref{kikslicesum}).
 We now recall the necessary and sufficient conditions on $\cB$ so that
 $\R_+^{m_1\times \ldots m_d}(\cB, \s_1,\ldots,\s_d)$
 contains a tensor $\cA$, which is positively diagonally equivalent to $\cB$.  For matrices, i.e. $d=2$,
 this problem was solved by
 Menon \cite{Men68} and Brualdi \cite{Bru68}.  See also \cite{MS69}.  For the special case of positive
 diagonal equivalence
 to doubly stochastic matrices see \cite{BPS66} and \cite{SK67}.
 The result of Menon was extended for tensors independently by Bapat-Raghavan \cite{BR89} and
 Franklin-Lorenz \cite{FL89}.  (See \cite{Bap82} and \cite{Rag84} for the special case where all the
 entries of $\cB$ are positive.)  In \cite{Fri11} we gave necessary and sufficient
 conditions for the solution of this problem:
 \begin{theorem}\label{maintheo}  Let $\cB=[b_{i_1,i_2,\ldots,i_{d}}]\in\R_+^{m_1\times\ldots\times m_d}$,
 $(d\ge 2)$, be a
 given nonnegative tensor with no $(k,i_k)$-zero slice.  Let $\s_k\in\R_+^{m_k}, k=1,\ldots,d$ be given
 positive vectors satisfying (\ref{compslccon}).
 Then there exists a nonnegative tensor $\cA\in \R_+^{m_1\times\ldots\times m_d}$, which is positive
 diagonally equivalent
 to $\cB$ and having each $(k,i_k)$-slice sum equal to $s_{k,i_k}$, if and only the following conditions  hold:
 The system of the inequalities and equalities for $\x_k=(x_{k,1},\ldots,\x_{k,m_k})
 \trans\in\R^{m_k}, k=1,\ldots,d$,
 \begin{eqnarray}\label{ineqcond}
 x_{1,i_1}+x_{2,i_2}+\ldots+x_{d,i_d}\le 0 \textrm{ if } b_{i_1,i_2,\ldots,i_d}>0,\\
 \label{eqcond}
 \s_k\trans \x_k=0 \textrm{ for } k=1,\ldots,d,
 \end{eqnarray}
 imply one of the following equivalent conditions
 \begin{enumerate}
 \item\label{eqcond1}  $x_{1,i_1}+x_{2,i_2}+\ldots+x_{d,i_d}= 0$ if $b_{i_1,i_2,\ldots,i_d}>0$.
 \item\label{eqcond2}
 $\sum_{b_{i_1,i_2,\ldots,i_d}>0} x_{1,i_1}+x_{2,i_2}+\ldots+x_{d,i_d}=0$.
\end{enumerate}
In particular, there exists at most one tensor $\cA\in \R_+^{m_1\times\ldots\times m_d}$ with $(k,i_k)$-slice
 sum $s_{k,i_k}$ for all $k,i_k$,  which is positive diagonally equivalent to $\cB$.
 \end{theorem}

 The above yields the following corollary.
 \begin{corollary}\label{mentens}  Let $\cB\in\R_+^{m_1\times\ldots\times m_d}$, $(d\ge 2)$, be a
 given nonnegative tensor with no $(k,i_k)$-zero slice.  Let $\s_k\in\R_+^{m_k}, k=1,\ldots,d$ be given
 positive vectors.  Then there exists a nonnegative tensor $\cC\in \R_+^{m_1\times\ldots\times m_d}$,
 which is positive diagonally equivalent
 to $\cB$ and each $(k,i_k)$-sum slice equal to $s_{k,i_k}$, if and only if there exists a nonnegative tensor
 $\cA=[a_{i_1,i_2,\ldots,i_{d}}]\in\R_+^{m_1\times\ldots\times m_d}$,
 having the same zero pattern as $\cB$, which satisfies (\ref{kikslicesum}).
 \end{corollary}
 For matrices, i.e. $d=2$, the above corollary is due Menon \cite{Men68}.
 For $d= 3$ this result is due to \cite[Thm 3]{BR89} and for $d\ge 3$ \cite{FL89}.
 Brualdi in \cite{Bru68} gave a nice and simple characterization for the set of nonnegative matrices, with
 prescribed zero pattern and with given positive row and column sums, to be not empty.
 It is an open problem to find an analog of Brualdi's results for $d$-mode tensors, where $d\ge 3$.

 Note that the conditions of Theorem \ref{maintheo}  are stated as a
 linear programming problem.  Hence the existence of a positive diagonally equivalent tensor $\cA$ can
 be determined in polynomial time \cite{Kar,Kha}.  If such $\cA$ exists,  it is shown in \cite{Fri11} that $\cA$ can be found
 by computing the unique minimal point of certain strictly convex functions $f$.
 Note that $\cB>0$ is always scalable as the tensor $\cA=b\s_{1}\otimes\cdots\otimes\s_d$ satisfies (\ref{kikslicesum}) for $b=(\1_{m_1}\trans\s_1)^{d-1}$.

 Identify  $\R^{m_1}\times\R^{m_2}\times\ldots\times\R^{m_d}$ with $\R^{n+d}$, where $n+d=\sum_{k=1}^d m_k$.
 We view $\x\in\R^{n+d}$ as a vector $(\x_1\trans,\ldots,\x_d\trans)\trans=(\x_1,\ldots,\x_d)$, where $\x_k\in\R^{m_k}, k\in[d]$.
 Let $\|\x\|:=\sqrt{\x\trans\x}$.  Define
 \begin{equation}\label{deffunct}
 \hat f(\x)=\hat f(\x_1,\ldots,\x_d):=\sum_{i_j\in[m_j],j\in[d]}
 b_{i_1,\ldots,i_d}e^{x_{1,i_1}+\ldots+x_{d,i_d}}.
 \end{equation}
 Clearly, $\hat f$ is a convex function on $\R^{n+d}$.
 Denote by $\U(\s_1,\ldots,\s_d)\subset \R^{n+d}$ the subspace of vectors  $(\x_1,\ldots,\x_d)$ satisfying the equalities (\ref{eqcond}).  Thus $\U(\s_1,\ldots,\s_d)=\rL(\s_1)\times\cdots\times \rL(\s_d)\equiv \R^n$. 
 In \cite{Fri11} we showed the following lemma:
 \begin{lemma}\label{critptf}  Let $\cB=[b_{i_1,i_2,\ldots,i_{d}}]\in\R_+^{m_1\times\ldots\times m_d}$,
 $(d\ge 2)$, be a given nonnegative tensor with no $(k,i_k)$-zero slice.  Let $\s_k\in\R_+^{m_k}, k=1,\ldots,d$
 be given positive vectors satisfying (\ref{compslccon}).  Then there exists a nonnegative tensor
 $\cA\in \R_+^{m_1\times\ldots\times m_d}$,
 which is positive diagonally equivalent to $\cB$ and having each $(k,i_k)$-slice sum equal to $s_{k,i_k}$, if
 and only the restriction of $\hat f$ to the subspace $\U(\s_1,\ldots,\s_d)$, denoted as $\tilde f$, has a critical point.
 \end{lemma}

 Denote by $\V(\s_1,\ldots,\s_d)$ the subspace of all vectors
 $(\x_1,\ldots,\x_d)$
 satisfying the condition \emph{\ref{eqcond1}} of Theorem \ref{maintheo}.
 Clearly, for each $\x\in\R^{n+d}$ the function $\hat f$ has a constant value $\hat f(\x)$ on the affine set $\x+
 \V(\s_1,\ldots,\s_d)$. Let $\V_0(\s_1,\ldots,\s_d)=\V(\s_1,\ldots,\s_d)\cap\U(\s_1,\ldots,\s_d)$  Hence, if $\et\in\U(\s_1,\ldots,\s_d)$ is a critical point of $\tilde f$
 then any point in $\et+\V_0(\s_1,\ldots,\s_d)$ is also a critical of  $\tilde f$. 
 Denote by $\V(\s_1,\ldots,\s_d)^{\perp}\subset \R^{n+d}$ the orthogonal complement of $\V(\s_1,\ldots,\s_d)$ in $\R^{n+d}$, and by
 $\V_0(\s_1,\ldots,\s_d)^{\perp}$, the orthogonal complement of
 $\V_0(\s_1,\ldots,\s_d)$ in $\U(\s_1,\ldots,\s_d)$.  In \cite{Fri11} we showed:
 \begin{lemma}\label{strictconvf}  Let $\cB=[b_{i_1,i_2,\ldots,i_{d}}]\in\R_+^{m_1\times\ldots\times m_d}$,
 $(d\ge 2)$, be a given nonnegative tensor with no $(k,i_k)$-zero slice.  Let $\s_k\in\R_+^{m_k}, k=1,\ldots,d$
 be given positive vectors satisfying (\ref{compslccon}).  Let $\U(\s_1,\ldots,\s_d),\V_0(\s_1,\ldots,\s_d),
 \V_0(\s_1,\ldots,\s_d)^{\perp}$ be defined as above.
 Then the restriction of $\tilde f$ to $\V_0(\s_1,\ldots,\s_d)^{\perp}$, denoted as $f$, is strictly convex.  That is, $H(f)$
 has positive eigenvalues at each point of  $\V_0(\s_1,\ldots,\s_d)^{\perp}$.
 \end{lemma}
 \begin{theorem}\label{eqivcondscal}  Let $\cB=[b_{i_1,i_2,\ldots,i_{d}}]\in\R_+^{m_1\times\ldots\times m_d}$,
 $(d\ge 2)$, be a given nonnegative tensor with no $(k,i_k)$-zero slice.  Let $\s_k\in\R_+^{m_k}, k=1,\ldots,d$
 be given positive vectors satisfying (\ref{compslccon}).  Then the following conditions are equivalent.
 \begin{enumerate}
 \item\label{eqivcondscal1}  $\tilde f$ has a global minimum.
 \item\label{eqivcondscal2}  $\tilde f$ has a critical point.
 \item\label{eqivcondscal3}  $\lim f(\x_l)=\infty$ for any sequence $\x_l\in\V_0(\s_1,\ldots,\s_k)^{\perp}$
 such that $\lim\|\x_l\|= \infty$.
 \item\label{eqivcondscal4}  The only $\x=(\x_1,\ldots,\x_d)\in\V_0(\s_1,\ldots,\s_k)^{\perp}$
 that satisfies (\ref{ineqcond}) is $\x=\0_n$.

 \end{enumerate}
 \end{theorem}
 \section{The scaling algorithm for tensors}\label{sec:tenscalgo}
 In this section we assume that a given $\cB=[b_{i_1,\ldots,i_d}]\in \R_+^{m_1}\times \cdots \times \R_+^{m_d}$ satisfies one of the equivalent conditions of Theorem \ref{eqivcondscal}.  Let $\cB(\x)=[b_{i_1,\ldots,i_d}e^{x_{1,i_1}+\cdots+x_{d,i_d}}]$.
 Hence $f$ has a unique minimum point $\x^\star\in \V_0(\s_1,\ldots,\s_d)^\perp$.
 We now describe our algorithm for finding $\x^\star$.  
 
 We first consider the case where $\V_0(\s_1,\ldots,\s_d)=\{\0\}$.  That is, the system of linear equations given by \eqref{eqcond} and by the conditions (1) of Theorem  \ref{eqivcondscal} has only the trivial solution $\x_1=\cdots=\x_d=\0$.  
 
This condition is satisfied if all the entries of $\cB$ are positive.  Indeed, assume that $\cB>0$.  Sum up the equations in condition (1) on $i_2,\ldots,i_d$ to deduce that $\x_1=t_1\1_{m_1}$.  Similarly, we deduce that $\x_j=t_j\1_{m_j}$ for all $j\in[d]$.  Furthermore  the $M=\prod_{j=1}^d m_j$ equations of (1) are equivalent to one equaiton:
 $t_1+\cdots+t_d=0$.  The conditions \eqref{eqcond} yield that $t_1=\cdots=t_d=0$.
 
 In this case $\tilde f= f$ is a function defined on $\U(\s_1,\ldots,\s_d)$.  We identify 
 $\U(\s_1,\ldots,\s_d)$ with $\R^n=\R^{m_1-1}\times\cdots\R^{m_d-1}$.
 Then our algorithm is applied straightforward as in the case $d=2$, which is described in Section \ref{sec:intro}:  
 Fix $\x^j$ and find the unique $\tilde \x_j=\tilde \x_j(\x^j)$ which satsfies the condition
 \begin{eqnarray*}
 \sum_{i_p\in[m_p],p\in [d]\backslash\{j\}} b_{i_1,\ldots,i_d}e^{x_{1,i_1}+\cdots+x_{d,i_d}}=s_{j,i_j}, \; i_j\in[m_j].
 \end{eqnarray*}
 Let $\x_j(\x^j)=\tilde\x_j -\frac{\s_j^\top \tilde \x_j}{\s_j^\top\1_{m_j}} \1_{m_j}$.  Clearly,
 $\x_j(\x^j)\in \rL(\s_j)$.  Hence $\x_j(\x^j)$ is  the critical point of the strict convex function $g_{\x^j}(\x_j)=f(\x^j,\x_j)$ on $\rL(\s_j)\equiv \R^{m_j-1}$.
 
 We now can apply Theorem \ref{convalgo}.  Our algorithm will converge to a unique minimal point $\x^\star\in\U(\s_1,\ldots,\s_d)$.   The tensor $\cB(\x^\star)$ will have its $d$ sum slices of the form $b\s_1,\ldots,b\s_d$ for some $b>0$.
 
 We now discuss the case where $\V_0(\s_1,\ldots,\s_d)$ is a nontrivial subspace of $\U(\s_1,\ldots,\s_d)$.  In that case we claim that our algorithm applies with a suitable modification.  First observe that 
 \begin{eqnarray*}
 \U(\s_1,\ldots,\s_d)=\V_0(\s_1,\ldots,\s_d)^\perp\oplus \V_0(\s_1,\ldots,\s_d).
 \end{eqnarray*}
 Let 
 \begin{eqnarray*}
  P:\R^n\to \V(\s_1,\ldots,\s_d)^\perp, \quad P_0:\U(\s_1,\ldots,\s_d)\to \V_0(\s_1,\ldots,\s_d)^\perp
 \end{eqnarray*}
 be the orthogonal projection on $\V(\s_1,\ldots,\s_d)^\perp$ and $\V_0(\s_1,\ldots,\s_d)^\perp$ respectively.  Then
 \begin{eqnarray*}
 \x=\y+ \z, \; \x=(\x_1,\ldots,\x_d),\;\y=(\y_1,\ldots,\y_d),\;\z=(\z_1,\ldots,\z_d).
 \end{eqnarray*}
 If $\x\in\R^n$ then  $\y=P\x\in \V(\s_1,\ldots,\s_d)^\perp,\z=(I-P)\x\in \V(\s_1,\ldots,\s_d)$.  If $\x\in \U(\s_1,\ldots,\s_d)$ then $\y=P_0\x\in \V_0(\s_1,\ldots,\s_d)^\perp$
 and $\z=(I-P_0)\x\in \V_0(\s_1,\ldots,\s_d)$.
 
Observe that $\hat f(\x+\z)=\hat f(\x)$ for $\x\in\R^n$ and $\z\in  \V(\s_1,\ldots,\s_d)$.  Hence
\begin{eqnarray*}
\hat f(\x)=\hat f(P\x), \; \nabla \hat f(\x)\in \V(\s_1,\ldots,\s_d)^\perp, \; \nabla \hat f(\x)=\nabla \hat f(P\x),\; \forall \x\in\R^n.
\end{eqnarray*}
 Similarly $\tilde f(\x+\z)=\tilde f(\x)$ for $\x\in\U(\s_1,\ldots,\s_d)$ and $\z\in  \V_0(\s_1,\ldots,\s_d)$.  Furthermore 
 \begin{eqnarray}\label{fP0id}\quad\quad
 \tilde f(\x)=\tilde f(P_0\x)=f(P_0\x),\; \nabla \tilde f(\x)=\nabla \tilde f(P_0\x), \; \forall \x\in \U(\s_1,\ldots,\s_d).
 \end{eqnarray}
 (The simplest way to show these identities is by considering an orthonormal basis in in $\U(\s_1,\ldots,\s_d)$ consisting of vectors in orthonormal bases of $\V_0(\s_1,\ldots,\s_d)^\top$ and $\V_0(\s_1,\ldots,\s_d)$.  Then change to a basis of $\U(\s_1,\ldots,\s_d)=\rL(\s_1)\times \cdots\times \rL(\s_d)$ which is a union of orthonormal bases of $\rL(\s_j)$ for $j\in[d]$.)
 
 Observe that for $\x\in \U(\s_1,\ldots,\s_d)$ the gradient $\nabla \tilde f(\x)$ is a subvector of $\nabla \hat f(\x)$, when we choose the corresponding coordinates in $\R^{m_1}\times\cdots\R^{m_d}$.  Similarly, for $\x\in\V_0(\s_1,\ldots,\s_d)^\perp$ the gradient $\nabla f(\x)$ is a subvector of $\nabla \tilde f(\x)$, if we choose the coordinates using the orthonormal bases in $\V_0(\s_1,\ldots,\s_d)^\perp$ and $\V_0(\s_1,\ldots,\s_d)$ respectively.  Moreover, the coordinates of $\nabla\tilde f(\x)$ corresponding to the chosen orthonormal basis in $\V_0(\s_1,\ldots,\s_d)$ are zero.
 Hence for $\x\in\V_0(\s_1,\ldots,\s_d)^\perp$ the gradient $\nabla f(\x)$ is obtained by deleting the zero coordinates of $\nabla \tilde f(\x)$ corresponding to the chosen orthonormal basis of $\V_0(\s_1,\ldots,\s_d)$.
 In particular we have the equality
\begin{eqnarray}\label{eqgradtilf}\quad\quad
\|\nabla f(\x)\|^2=\|\nabla \tilde f(\x)\|^2=\sum_{j=1}^d \|\nabla _j \tilde f(\x)\|^2 \textrm{ for } \x\in\V_0(\s_1,\ldots,\s_d)^\perp.
 \end{eqnarray}
 \begin{lemma}\label{propWj}  
Assume that a given $\cB=[b_{i_1,\ldots,i_d}]\in \R_+^{m_1}\times \cdots \times \R_+^{m_d}$ satisfies the assumptions of Theorem \ref{eqivcondscal} and one of its equivalent conditions .  
Suppose furthermore that $\dim \V_0(\s_1,\ldots,\s_d)>0$.  For $j\in[d]$ let  $\W_j$ be the following subspace of $\V_0(\s_1,\ldots,\s_d)^\perp$: $\{\bw=P_0(\0,\x_j), \x_j\in\rL(\s_j)\}$.  Then
\begin{enumerate}
\item The dimension of $\W_j$ is $m_j-1$ for $j\in[d]$.
\item $\W_1+\cdots+\W_d=\V_0(\s_1,\ldots,\s_d)^\perp$.  Furthermore, $\V_0(\s_1,\ldots,\s_d)^\perp$ is not a direct sum of $\W_1,\ldots,\W_d$.
\item  Let $\x\in \V_0(\s_1,\ldots,\s_d)^\perp$ and $j\in[d]$.  Choose an orthonormal basis in $\W_j$ and denote by $\nabla f_{\W_j}(\x)$ the gradient of $f$ with respect to the chosen orthonormal basis of the subspace $\W_j$.  Then 
\begin{eqnarray}\notag
&&\|\nabla_j\tilde f(\x)\|\le \|\nabla _{\W_j}f(\x)\| \textrm{ for all }j\in[d],\\
&&\|\nabla f(\x)\|^2\le \sum_{j=1}^d \|\nabla_{\W_j}f(\x)\|^2.
\label{intnabj}
\end{eqnarray}
\end{enumerate}
\end{lemma}
\begin{proof} (1) In view of the assumption \eqref{admjas} it follows that $\dim \rL(\s_j)=m_j-1\ge 1$. Assume to the contrary that $\dim \W_j<m_j-1$,  Then there exists $\x_j\in \rL(\s_j)\setminus\{\0\}$ such that $P_0(\0,\x_j)=\0$.  Use the first equality 
of\eqref{fP0id} to deduce that $\tilde f((\0,t\x_j))=f(P_0(\0,t\x_j))=f(\0)$ for each $t\in\R$.
As $\cB$ is a nonnegative tensor with no $(k,i_k)$-zero slice it follows that
\begin{eqnarray*}
\tilde f((\0,t\x_j))=\sum_{i=1}^{m_j}e^{tx_{j,i}}a_{j,i}, \quad \x_j=(x_{j,1},\ldots,x_{j,m_j})^\top, a_{j,i}>0 \textrm{ for } i\in[m_j].
\end{eqnarray*}
As $\x_j\ne \0$ the above function of $t$ can't be a constant function.  Hence $\dim\W_j=m_j-1$.

\noindent
(2) Let $\x=(\x_1,\ldots,\x_d)\in \rL(\s_1)\times\cdots\times \rL(\s_d)$.  Then $\x=\sum_{j=1}^d (\0,\x_j)$.  Hence $P_0\x=\sum_{j=1}^d P_0(\0,\x_j)$.  Therefore  $\W_1+\cdots+\W_d=\V_0(\s_1,\ldots,\s_d)^\perp$.  Clearly 
\begin{eqnarray*}
&&\dim \V_0(\s_1,\ldots,\s_d)^\top=\dim \U(\s_1,\ldots,\s_d)-\dim \V_0(\s_1,\ldots,\s_d)
<\\
&&\dim \U(\s_1,\ldots,\s_d)=\sum_{j=1}^d (m_j-1).
\end{eqnarray*}
Hence $\V_0(\s_1,\ldots,\s_d)^\perp$ is not a direct sum of $\W_1,\ldots,\W_d$. 

\noindent
(3) If $\nabla_j\tilde f(\x)=\0$ the inequality \eqref{intnabj} trivially holds.  Assume that $\nabla_j\tilde f(\x)\ne\0$.  Let $\bw_j=\frac{1}{\|\nabla_j\tilde f(\x)\|}\nabla_j\tilde f(\x)$.
Then $\|\nabla_j\tilde f(\x)\|=\nabla \tilde f(\x)^\top (\0,\bw_j)$.   Let
\begin{eqnarray*}
&&\bu_j=P_0(\0,\bw_j)\in \W_j, \quad \bv_j=(I-P_0)(\0,\bw_j)\in \V_0(\s_1,\ldots,\s_d), \\
&&(\0,\bw_j)=\bu_j+\bv_j, \quad \|\bu_j\|^2 +\|\bv_j\|^2=\|\bw_j\|^2=1.
\end{eqnarray*} 
As $\nabla \tilde f(\x)^\top \bv_j=0$ we deduce that 
\begin{eqnarray*}
&&\|\nabla_j \tilde f(\x)\|=\nabla \tilde f(\x)^\top (\0,\bw_j) =\nabla \tilde f(\x)^\top \bu_j=\nabla_{\W_j} f(\x)^\top \bu_j\\
&&\le \|\nabla_{\W_j} f(\x)\| \|\bu_j\|\le \|\nabla_{\W_j} f(\x)\|.
\end{eqnarray*}
Use \eqref{eqgradtilf} and the above inequalities to deduce \eqref{intnabj}.
\end{proof}
We now give the modified algorithm:
 
 \noindent
$\quad\quad$ $\;\;$\textbf{Modified algorithm} 
\begin{itemize}
\item[] Choose $\x_0\in\V_0(\s_1,\ldots,\s_d)^\perp$. 
\item[] for $k:=0,1,2, \ldots$
\item[] $\quad$$j\in\argmax\{\|\nabla_{\W_l} f(\x_k)\|, l\in[d]\}$
\item[] $\quad$$\x_{k+1}=P_0(\x_k^j, \x_j(\x_k^j))$
\item[]end
\end{itemize} 
 
We explain and justify the modified algorithm.
 View $\x_0$ as a point in $\U(\s_1,\ldots,\s_d)$. 
 Then $\x_j(\x^j_0)$ is a critical point of the strict convex function $g_{\x^j_0}(\x_j)=\tilde f(\x^j_0,\x_j), \x_j\in \rL(\s_j)$ as in the beginning of this section.  Let $ \x_1'=\x_0 +(\0,\x_j(\x_0^j)-\x_{j,0})$.  Clearly $\x_1'\in\U(\s_1,\ldots,\s_d)$.   Note that  $\nabla_j \tilde f(\x_1')=0$.   Hence $\tilde f(\x_1 +(\0,\x_j))\ge f(\x_1)$ for each $\x_j\in\rL(\s_j)$.  
 Let $\x_1=P_0\x_1'=\x_0+P_0(0,\x_j(\x_0^j)-\x_{j,0})$.  The first equality of \eqref{fP0id} yields:
\begin{eqnarray*}
&&f(\x_1)=\tilde f(\x_1)=\tilde f(P_0\x'_1)=\tilde f(\x_1)\le 
\tilde f(\x_1'+(\0,\x_j))=\\
&&\tilde f(P_0(\x'_1+(\0,\x_j))= f(\x_1+P_0(\0,\x_j)) \textrm{ for all } \x_j\in\rL(\s_j).
\end{eqnarray*}
Hence $\nabla_{\W_j} f(\x_1)=0$.  Therefore $\x_1$ is the minimum of $f$ on the affine space $\x_0+\W_j$.  Inequality \eqref{intnabj} yields that $\|\nabla_{\W_j} f(\x_0)\|\ge \frac{\|f(\x_0)\|}{\sqrt{d}}$, as in the case of the original algorithm.  As $\nabla_{\W_j} f(\x_1)=0$ we deduce from \eqref{intnabj} that $\|\nabla_{\W_j} f(\x_k)\ge \frac{\|f(\x_k)\|}{\sqrt{d-1}}$ for $k=1$.  Same inequality holds for all $k\ge 1$.  Hence Theorem \ref{convalgo} applies in this case too.
\section{A generalization of discrete Schr\"odinger's bridge problem}\label{sec:sbp}
The classical Schr\"odinger bridge problem, studied by Schr\"odinger in \cite{Sch31, Sch32}, seeks the most likely probability law for a diffusion process, in path space, that matches marginals at two end points in time.  
The discrete version of Schr\"odinger's bridge problem for Markov chains can be stated as follows \cite{GP15, Fri17}:  
\begin{problem}\label{SBPMC}  Let $A\in\R_+^{n\times n}$ be a column stochastic matrix.  Assume that $\ba,\bb$ are two positive probability vectors.  Does there exists a scaling of $A$, denoted as $B$, such that $B$ is column stochastic and $B\ba=\bb$?
\end{problem}

We give a necessary and sufficient condition for a solution to generalized  Schr\"odinger's bridge problem:
\begin{theorem}\label{GSBMB}  Let $A\in\R^{m\times n}_+$ be a given matrix.
Assume that $\bb\in \R^m,\ba,\bc\in\R^n$ be given positive vectors that satisfy $\bc^\top \ba=\1_m^\top \bb$.  Then there exists a scaling of $A$, denoted as $B$, such that 
\begin{eqnarray}\label{GSBMB1}
B\ba=\bb, \quad B^\top \1_m=\bc,
\end{eqnarray}
if and only if the following conditions holds:
There exists $C\in\R^{m\times n}_+$ with the same $0$-pattern as $B$ that satisfies \eqref{GSBMB1}.  If this condition holds then $B$ is unique and can be found by the modified algorithm.
\end{theorem}
\begin{proof}  Assume that $\ba=(a_1,\ldots,a_n)^\top, \bc=(c_1,\ldots,c_n)^\top$.
Denote by $D(\ba)\in\R^{n\times n}$ the diagonal matrix whose diagonal entries are the coordinates of $\ba$.  Let $\tilde A=AD(\ba)$ and consider the scaling of $B=D_1\tilde AD_2$ with the row sum $\bb$ and column sum $\bc\circ\ba=(c_1a_1,\ldots,c_na_n)^\top$.  Note that condition $\1_n^\top (\bc\circ \ba )=\1_m^\top \bb$ is the condition $\bc^\top\ba=\1_m^\top \bb$.  Next observe that this scaling of $\tilde A$ is equivalent to the scaling of $A$ which satisfies \eqref{GSBMB1}.  The result of \cite{Men68} yields that $B$ exists if and only if there exists $C\in\R^{m\times n}_+$ with the same $0$-pattern as $B$ that satisfies \eqref{GSBMB1}.   Use the modifed algorithm to find the scaling of $\tilde A$.
\end{proof}

\bibliographystyle{plain}

\begin{thebibliography}{MMM}
\bibitem{ALOW17} Z. Allen-Zhu, Y. Li,  R. Oliveira and A. Wigderson, Much faster algorithms for matrix scaling, arXiv:1704.02315.

\bibitem{Bap82} R.B. Bapat $D_1AD_2$ theorems for multidimensional matrices,  \emph{Linear Algebra Appl.}
 48 (1982), 437--442.
 
 \bibitem{BR89} R.B. Bapat and T.E.S. Raghavan, An extension of a theorem of Darroch and Ratcliff
 in loglinear models and its application to scaling multidimensional matrices,  \emph{Linear Algebra Appl.}
 114/115 (1989), 705-715.

 \bibitem{Bru68} R.A. Brualdi, Convex sets of nonnegative matrices, \emph{Canad. J. Math} 20 (1968),
 144-157.

 \bibitem{BPS66} R.A. Brualdi, S.V. Parter and H. Schneider,
 The diagonal equivalence of a nonnegative matrix to a stochastic
 matrix, \emph{J. Math. Anal. Appl.}  16 (1966), 31--50.
 \bibitem{BLS15} S. Bubeck, Y. T. Lee, M. Singh, A geometric alternative to Nesterov's accelerated gradient descent,  arXiv:1506.08187.

 \bibitem{FL89} J. Franklin and J. Lorenz, On the scaling of multidimensional matrices,
 \emph{Linear Algebra Appl.} 114/115 (1989), 717-735.
 
\bibitem{Fri11} S. Friedland, Positive diagonal scaling of a nonnegative tensor to one with prescribed slice  sums,  \emph{Linear Algebra Appl.}, vol. 434 (2011), 1615-1619.

\bibitem{Fri17} S. Friedland, On Schr$\ddot{\textrm{o}}$dinger's bridge problem, SB MATH, 2017, 208 (11), 139--156,

\bibitem{GP15}  T.T.~ Georgiou and M.Pavon, Positive contraction mappings for classical and quantum Schr\"odinger systems,
\emph{J. Math. Physics}, 56 (2015), 1--24.

\bibitem{HS52} M.R. Hestenes and E. Stiefel. Methods of conjugate gradients for
solving linear systems,  \emph{Journal of Research of the National Bureau of Standards}, 49 (1952), 409--436.

 \bibitem{Kar} N.K. Karmakar, A new polynomial algorithm for
 linear programming, \emph{Combinatorica} 4 (1984), 373-395.

 \bibitem{Kha} L.G. Khachiyan, A polynomial algorithm in linear
 programming, \emph{Doklady Akad. Nauk SSSR} 224 (1979),
 1093-1096.  English Translation: \emph{Soviet Mathematics
 Doklady} 20, 191-194.
 
 \bibitem{Men68} M.V. Menon, Matrix links, an extremisation problem and the reduction of a nonnegative
 matrix to one with with prescribed row and column sums, \emph{Canad. J. Math} 20 (1968),
 225-232.

 \bibitem{MS69} M.V. Menon and H. Schneider, The spectrum of a nonlinear operator associated with a matrix,
 \emph{Linear Algebra Appl.} 2 (1969), 321-334.

 \bibitem{Rag84} T.E.S. Raghavan, On pairs of multidimensional matrices, \emph{Linear Algebra Appl.}
 62 (1984), 263-268.
 
 \bibitem{Sch31}
E. Schr\"odinger, \"Uber die Umkehrung der Naturgesetze,
\emph{Sitzungs- berichte der Preuss Akad. Wissen. Berlin}, Phys. Math. Klasse
(1931), 144--153

\bibitem{Sch32} E. Schr\"odinger, Sur la th\'eorie relativiste de l'\'electron et l'interpr\'etation de la m\'ecanique quantique
\emph{Ann. Inst. H. Poincar\'e 2} 2 (4) (1932), 269--310.
 
 \bibitem{Sin64}  R.A. Sinkhorn,  A relationship between arbitrary positive matrices and doubly stochastic matrices, \emph{Ann. Math. Statist.} 35 (1964),  876--879.

 \bibitem{SK67} R. Sinkhorn and P. Knopp, Concerning nonnegative matrices and doubly stochastic matrices,
 \emph{Pac. J. Math.} 21 (1967), 343-348.





 \end{thebibliography}

\end{document}